\documentclass[11pt,english]{amsart}
\usepackage{euscript,amsthm,amssymb,amsbsy,amsfonts}
\usepackage{babel}
\usepackage{a4wide}
\usepackage[all]{xy}
\usepackage{graphicx}

\numberwithin{equation}{section} 
\numberwithin{figure}{section}

\newtheorem{stuff}{Stuff}[section]
\newtheorem{theorem}[stuff]{\bf Theorem}
\newtheorem{proposition}[stuff]{\bf Proposition}
\newtheorem{lemma}[stuff]{\bf Lemma} 
\newtheorem{corollary}[stuff]{\bf Corollary} 
 
\newtheorem{question}[stuff]{\bf Question} 
\newenvironment{definition}{%
\vskip1ex\refstepcounter{stuff}\trivlist \itemindent 0pt
\item[\hskip\labelsep\bf Definition \thestuff.]%
\ignorespaces}{\endtrivlist\vskip1ex}%
\newenvironment{remark}{%
\vskip1ex\refstepcounter{stuff}\trivlist \itemindent 0pt 
\item[\hskip\labelsep\bf Remark \thestuff.]%
\ignorespaces}{\endtrivlist\vskip1ex}%

\newtheorem{s-theorem}[sstuff]{\bf Theorem} 
\newtheorem{s-proposition}[sstuff]{\bf Proposition} 
 
\let\rar\rightarrow 
\let\lar\longrightarrow

\let\hra\hookrightarrow 
\let\mt\mapsto 
\let\lmt\longmapsto 
 
\let\xrar\xrightarrow 
 
\font\tenmsa=msam10 %
\newcommand\hdashpiece{%
{\vrule height2.75pt depth-2.35pt width2.3pt \kern1.7pt}}%
\newcommand\hdashpieces{%
{\hdashpiece\hdashpiece\hdashpiece\hdashpiece}}%
\newcommand\dashto{\mathrel{%
\hdashpiece\hdashpiece\kern-0.4pt\hbox{\tenmsa K}}}%
\newcommand\dashar{\mathrel{%
\hdashpieces\kern-0.4pt\hbox{\tenmsa K}}}%
 
\let\euf\EuScript 
\let\cal\mathcal 
\let\mbb\mathbb

\DeclareFontFamily{OT1}{rsfs}{} 
\DeclareFontShape{OT1}{rsfs}{n}{it}{<->rsfs10}{} 
\DeclareMathAlphabet{\crl}{OT1}{rsfs}{n}{it} 
 
\let\unbar\underbar

\let\nit\noindent 
 
\let\disp\displaystyle 
\let\srel\stackrel

\let\veps\varepsilon 
\let\disc\triangle

\newcommand\del{{\partial}} 
 
\newcommand\rd{{\rm d}}

\newcommand\Pic{\mathop{\rm Pic}\nolimits} 
 
\newcommand\Ker{{\rm Ker}} 
 
\newcommand\invq{{\slash\kern-0.65ex\slash}}

\let\l\lambda 
 
\let\d\delta 
 
\let\O\Omega 
\let\si\sigma

\let\les\leqslant 
\let\ges\geqslant

\newcommand\bone{{1\kern-0.57ex\rm l}} 
 
\newcommand\pr{\mathop{\rm pr}\nolimits}

\newcommand\mx{{\rm max}} 
\newcommand\Img{{\rm Image}}

\newcommand{\tS}{{\tilde S}} 
\newcommand{\res}{{\rm res}}

\begin{document} 
 
\title{Modular properties  
of nodal curves on $K3$ surfaces} 
\author{Mihai Halic} 
\address{Dept. of Mathematics and Statistics,  
KFUPM P.O. Box 5046, Dhahran 31261, Saudi Arabia} 
\subjclass[2000]{Primary: 14H10, 14J28; Secondary: 14D15}

\begin{abstract} 
In this paper we partially address two questions which have been raised  
in \cite{fkps}:  
\begin{enumerate} 
\item[--] The first is a rigidity property for pairs $(S,C)$ consisting of  
a general projective $K3$ surface $S$, and a curve $C$ obtained as the  
normalization of a nodal, hyperplane section of $S$. We prove that a  
non-trivial deformation of such a pair $(S,C)$ induces a non-trivial  
deformation of $C$; 
\item[--] The second question concerns the Wahl map of curves $C$ as above.  
We prove that the Wahl map of the normalization of a nodal curve contained  
in a general projective $K3$ surface is non-surjective.  
\end{enumerate} 
In both cases, we impose upper bounds on the number of nodes of the  
hyperplane section. 
\end{abstract} 
 
\maketitle 
\markboth{\sc Mihai Halic}%
{\sc Modular properties of nodal curves on $K3$ surfaces} 
 
\section*{Introduction} 
 
Curves on $K3$ surfaces have been investigated from various points  
of view, and there is an extensive literature concerning their  
properties.  
Most attention has been payed to the {\it smooth} curves.  
In a series of articles Mukai studied the properties of the morphism  
\begin{equation*}{\tag{$\mu$}} 
\left\{ 
(S,C)\;\biggl| 
\begin{array}{l} 
S\text{ is a general projective $K3$ surface, and } 
\\  
C\text{ is a smooth hyperplane section of genus }g.  
\end{array} 
\biggr. 
\right\} 
\srel{\mu}{\lar}  
\crl M_g, 
\end{equation*} 
which associates to a pair $(S,C)$ consisting of a general, projective $K3$  
surface the class of the curve $C$ in the Deligne-Mumford space.  He  
proved in \cite{mu} that the morphism $(\mu)$ is finite for $g\ges 13$,  
and then he went on proving that it is actually birational  
(see \cite[theorem 1.2]{mu2}).  
These topics are nicely surveyed in \cite{mu3} and \cite{beau}.  
\smallskip 
 
By contrast {\it nodal} curves on $K3$ surfaces have received somewhat less  
attention. The existence of nodal curves on $K3$ surfaces has been addressed  
in \cite{mm}, and later on generalized in \cite{chen}.  
The deformation theory of nodal  curves on $K3$ surfaces has been treated in  
\cite{ta}, and recently in \cite{fkps}. The goal of this paper is to (partially)  
address the following two questions raised in \cite{fkps}:\smallskip 
  
\begin{enumerate} 
\item The first problem (see 5.7(ii) in {\it loc.\,cit.}) concerns the  
finiteness of the forgetful morphism $(\mu)$, where one considers now  
pairs $(S,C)$ such that $C$ is the normalization of a nodal curve on $S$.  
\item The second problem is to find obstructions for embedding nodal curves  
into $K3$ surfaces.  
More precisely, a result due to Wahl says that for a smooth curve $C$  
lying on a projective $K3$ surface, the homomorphism  
\begin{equation*} 
w_C:\hbox{$\overset{2}\bigwedge$}\,  
H^0(C,\cal K_C)\rar H^0(C,\cal K_C^{\otimes 3}) 
\end{equation*} 
is non-surjective (see \cite{wa,bm}). The question raised in  
\cite[question 5.5]{fkps} is the following:  
suppose that $C$ is the normalization of a nodal curve on a projective  
$K3$ surface.  
Is it true that the homomorphism $w_C$ is still non-surjective?  
\end{enumerate}\smallskip 
 
For these questions we have the following two answers:\medskip  
 
\nit{\bf Theorem}\quad{\it  
For two positive integers $n$ and $d$ (subject to the inequalities  
below), we define:\\[1ex] 
\centerline{$ 
\d_\mx(n,d)=\left\{ 
\begin{array}[c]{ll} 
\bigl\lfloor\frac{n}{2}\bigr\rfloor-25& 
\text{ if $d=1$ and $n\ges 50$}; 
\\[1.5ex] 
2n-27& 
\text{ if $d=2$ and $n\ges 14$};  
\\[1.5ex] 
2(n-1)(d-1)-25& 
\text{ if $d\ges 3$ and $n=11$ or $n\ges 13$}. 
\end{array}\right.$}\smallskip 
 
\nit{\rm (i)} The forgetful morphism\\[1ex]  
$ 
\left\{\kern-.5ex 
(S,C,u) 
\left| 
\begin{array}{l} 
(S,\cal A)\in\crl K_n,\,C\in\crl M_{n-\d},\text{ and }\,  
u:C\rar S\text{ is a morphism} 
\\  
\text{s.t. }u_*C\hra S  
\text{ is a reduced, nodal curve with }\d\text{ nodes,} 
\\  
\text{which belongs to the linear system }|d\cal A|. 
\end{array} 
\right.\kern-1ex 
\right\} 
\srel{\mu}{\lar}\crl M_{1+(n-1)d^2-\d}$\\[1ex] 
is generically finite onto its image.\smallskip  
 
\nit{\rm (ii)} Suppose $(S,\cal A)$ is a polarized $K3$ surface,  
with $\Pic(S)=\mbb Z\cal A$, $\cal A^2=2(n-1)$,  
and consider a nodal, hyperplane section $\hat C$ of $S$ of degree  
$d$, with $\d$ nodes. Assume that  
$$ 
\d\les\min\left\{\d_\mx(n,d),\frac{(n-1)d^2-1}{3}\right\}. 
$$ 
Let $C$ be the normalization of $\hat C$.  
Then the Wahl map of $C$ is not surjective. 
}\medskip 
 
A remark concerning the upper bound appearing in (ii) above:  
there are few articles discussing the surjectivity properties of the  
normalization of nodal curves on surfaces.  
Actually, the author of this paper could find only the reference  
\cite{clm}, which deals with the surjectivity of the Wahl map of  
{\em plane} nodal curves.  
In that reference, the authors impose an upper bound on the number  
of nodes too.\smallskip  
 
This article is structured as follows:\smallskip  
 
-- In the first section we briefly recall basic facts concerning the  
deformation theory of curves on surfaces, and fix the notations used  
throughout the article.\smallskip 
 
-- The second section contains our main technical tool used for  
answering the two above mentioned questions. It is well-known  
that the tangent bundle of any $K3$ surface $S$ is stable.  
In proposition \ref{tech} we give an effective upper bound for the number  
of nodes of a nodal curve $\hat C\hra S$, such that the pull-back  
of the tangent bundle of $S$ to the normalization of $\hat C$ is  
still stable.\smallskip  
 
-- The third and the fourth sections contain the proofs of the first,  
respectively the second main result.\smallskip


\section{Description of the problem}{\label{sect:setup}} 
 
Throughout the article we will work over the field $\mbb C$  
of complex numbers.  
Most of the material appearing in this section is contained in the articles  
\cite{beau} and \cite{fkps}.  
Here we will introduce only those objects, and recall those properties,  
which are essential for our presentation.  
 
\begin{definition} 
\begin{enumerate} 
\item We say that a polarized $K3$ surface $(S,A)$ is  
{\it Picard general} if  
$$ 
\Pic(S)=\mathbb Z\cal A,\text{ with $\cal A\rightarrow X$ ample.} 
$$ 
In this case the self-intersection number $\cal A^{2}=2(n-1)$,  
with $n\ges 3$,  
and the linear system $|\cal A|$ induces an embedding $S\hra\mbb P^n$.  
 
\item We say that a morphism $\euf{S}\stackrel{\pi}{\rightarrow}\disc$ 
between irreducible algebraic varieties is a family of Picard general, 
polarized $K3$ surfaces, if there is a relatively ample line bundle 
$\mathcal{A}\rightarrow\euf{S}$ such that fibres $(S_{t},\cal A_{t})$,  
$t\in\disc$, are Picard general $K3$ surfaces. In this case the function  
$t\mt\cal A_{t}^{2}$ is constant.  
\end{enumerate} 
\end{definition} 
 
\begin{theorem} 
\begin{enumerate} 
\item Let $\crl K_n$ be the set of Picard general, polarized $K3$ surfaces.  
Then $\crl K_n$ can be endowed with the structure of a smooth stack,  
whose local charts are given by the local Kuranishi models of its points.    
\item For any $g\ges 1$, let $\crl M_g$ be the Deligne-Mumford stack  
of smooth and irreducible curves of genus $g$.  
For $d\ges 1$ and $0\les\d\les (n-1)d^2$, we define\\[1ex]  
$\crl V_{n,\d}^d:= 
\left\{ 
(S,C,u) 
\left| 
\begin{array}{l} 
(S,\cal A)\in\crl K_n,\,C\in\crl M_{n-\d},\text{ and }\,  
u:C\rar S\text{ is a morphism} 
\\  
\text{s.t. }u_*C\hra S  
\text{ is a reduced, nodal curve with }\d\text{ nodes,} 
\\  
\text{which belongs to the linear system }|d\cal A|. 
\end{array} 
\right.\! 
\right\}$\\[1ex] 
Then $\crl V_{n,\d}^d$ can be endowed with the structure of  
an analytic stack, which admits two forgetful morphisms 
\begin{eqnarray}{\label{morph}} 
\xymatrix@C=1.5em@R=2em{ 
&\crl V_{n,\d}^d\ar[dl]_-\mu\ar[dr]^-\kappa& 
\\ 
\crl M_{g(d)-\d}&&\crl K_n 
} 
&\text{with }\;g(d):=1+(n-1)d^2. 
\end{eqnarray} 
\item There is a non-empty open subset ${\crl K_n}^\circ\subset\crl K_n$ such  
that ${({\crl V}_{n,\d}^d)}^\circ:=\kappa^{-1}\bigl({\crl K_n}^\circ\bigr)$ is  
smooth, the projection ${({\crl V}_{n,\d}^d)}^\circ\rar{\crl K_n}^\circ$ is  
submersive, and all the irreducible components of ${({\crl V}_{n,\d}^d)}^\circ$  
are $19+g(d)-\d$ dimensional. 
\end{enumerate} 
\end{theorem} 
Note that for $d=1$ we recover the situation studied in \cite[section 4]{fkps}:  
${\bigl(\crl V_{n,\d}^1\bigr)}^\circ$ coincides with the stack $\crl V_{n,\d}$   
introduced in {\it loc.\,cit.}, definition 4.3. 
 
\begin{proof} 
(i) The detailed construction can be found for instance in  
\cite[chap. VIII, sect. 12]{bpv}.  
 
\nit(ii) The analytic stack structure is obtained  
as follows: for a family $(\euf S,\cal A)\srel{\pi}{\rar}\disc$  
of Picard general $K3$ surfaces, $\crl V_{n,\d}^d(\euf S)$ is  
naturally an open subscheme of the Kontsevich-Manin space of  
stable maps $\crl M_{g(d)-\d}(\euf S;\beta)$, with suitable  
$\beta\in H_2(\euf S;\mbb Z)$ such that $\pi_*\beta=0$.  
 
\nit(iii) The proof is {\it ad litteram} the same as that of  
\cite[proposition 4.8]{fkps}. According to \cite{chen}, there  
is a non-empty open subset  
${\crl K_n}^\circ\subset\crl K_n$ such that  
$$ 
\forall\,(S,\cal A)\in{\crl K_n}^\circ,\; 
\text{ the linear system $|d\cal A|$ contains irreducible,  
nodal curves with $\d$ nodes.} 
$$  
Consider a point $(S,C,u)\in\kappa^{-1}\bigl({\crl K_n}^\circ\bigr)$, and  
denote $\hat C:=u_*C$.  
Then the short exact sequence\\[1ex]  
\centerline{$ 
0\lar \cal T_S\langle\hat C\rangle:=\Ker(\hat\l) 
\srel{\iota}{\lar}\cal T_S\srel{\hat\l}{\lar} 
u_*\bigl(\;\underbrace{ u^*\cal T_S/\cal T_C }_{\cong\;\cal K_C}\;\bigr)\lar 0 
$}\\[1ex] 
induces the long exact sequence in cohomology: 
$$ 
0\rar H^0(C,\cal K_C)\rar H^1(S,\cal T_S\langle\hat C\rangle) 
\srel{H^1(\iota)}{\lar}  
\underbrace{\,H^1(S,\cal T_S)\,}_{\cong\;k^{20}} 
\srel{H^1(\hat \l)}{\lar}  
\underbrace{\,H^1(C,\cal K_C)\,}_{\cong\;k} 
\rar H^2(S,\cal T_S\langle\hat C\rangle)\rar 0 
$$ 
The cohomology group $H^1\bigl(S,\cal T_S\langle\hat C\rangle\bigr)$ is  
naturally isomorphic to the Zariski tangent space $\cal T_{\crl V^d_{n,\d},(S,C,u)}$  
and the homomorphism $H^1(\iota)$ can be identified with the differential  
of $\kappa$ at $(S,C,u)$ (see diagram \eqref{deform} below). Hence:\\[1ex]  
\centerline{$ 
\begin{array}[t]{lll} 
\Img\bigl(H^1(\iota)\bigr)\subset\cal T_{\crl K_n,[S]}& 
\Rightarrow& 
\dim\bigl(\Img\bigl(H^1(\iota)\bigr)\bigr)\les 19, 
\\  
\Img\bigl(H^1(\iota)\bigr)=\Ker\bigl(H^1(\hat \l)\bigr)& 
\Rightarrow& 
\dim\bigl(\Img\bigl(H^1(\iota)\bigr)\bigr)\ges 19. 
\end{array} 
$}\\[1ex]  
We deduce that $\Img(\rd\kappa_{(S,C,u)})= 
\Img\bigl(H^1(\iota)\bigr)=\cal T_{\crl K_n,[S]}$,  
and $H^2(S,\cal T_S\langle\hat C\rangle)=0$.  
\end{proof} 
 
The Zariski tangent space of $\crl V_{n,\d}^d$ is described  
in \cite[section 4]{fkps}.  
Consider a triple $(S,C,u)\in\crl V_{n,\d}^d$, and let  
$\tilde S\srel{\si}{\rar}S$ be the blow-up of $S$ at the  
$\d$ double points of $\hat C=u(C)$.  
Then the morphism $u$ can be lifted to a morphism $\tilde u$ into  
$\tilde S$, which is a closed embedding: 
$$ 
\xymatrix@R=2em@C=4em{ 
&\tilde S\ar[d]^-\si 
\\  
C\ar[ur]^{\tilde u}\ar[r]^u&S. 
} 
$$ 
The infinitesimal deformations of $(S,C,u)$ are controlled by  
the (locally free) sheaf $(\si^*\cal T_S)\langle C\rangle$ defined  
in the diagram below: 
\begin{equation}{\label{deform}} 
\xymatrix@C=1.35em@R=2em{ 
&0\ar[d]&0\ar[d]&&& 
\\  
&(\si^*\cal T_S)(-C)\ar@{=}[r]\ar[d] 
& 
(\si^*\cal T_S)(-C)\ar[d]&&& 
\\  
0\ar[r]& 
(\si^*\cal T_S)\langle C\rangle:=\Ker(\l)\ar[r]\ar[d]^-{r}& 
\si^*\cal T_S\ar[r]^-{\l}\ar[d]& 
\cal N_u\ar[r]\ar[d]& 
0& 
{}\save[]+<6em,-1.7ex>*\txt<14em>{%
$\text{with } 
\cal N_u\!:=\tilde u_*\bigl(\, 
\underbrace{\tilde u^*\cal T_S/\cal T_C}_{\cong\;\cal K_C} 
\,\bigr).$} 
\restore 
\\  
0\ar[r]& 
\cal T_C\ar[r]\ar[d]& 
u^*\cal T_S\ar[r]\ar[d]& 
\cal N_u\ar[r]\ar[d]& 
0&\\  
&0&0&0&& 
} 
\end{equation} 
More precisely, the Zariski tangent space to $\crl V_{n,\d}^d$  
at $(S,C,u)$ is isomorphic to  
$$ 
H^1(\tilde S,(\si^*\cal T_S)\langle C\rangle) 
\cong H^1(S,\cal T_S\langle\hat C\rangle). 
$$ 
We are finally in position to precise the topic of this article.  
The first issue is: 
 
\begin{question}{\label{issue1}} 
Is the morphism $\crl V_{n,\d}^d\srel{\mu}{\lar}\!\crl M_{g(d)-\d}$  
(generically) finite?  
\end{question} 
 
The second issue is the following: one can easily see that   
$\,\dim\bigl(\crl V_{n,\d}^d\bigr)<\dim\bigl(\crl M_{g(d)-\d}\bigr)\,$  
for $n$ sufficiently large.  
\begin{question}{\label{issue2}} 
Is there any obstruction for a point $C\in\crl M_{g(d)-\d}$ to lay in  
the image of $\mu$?  
\end{question}

 
\section{A vanishing result}{\label{sct:0}} 
 
In this section we prove the main technical ingredient needed for  
our approach to the question \ref{issue1}. It is an application of  
Bogomolov's effective restriction theorem for stable vector bundles  
over surfaces (see \cite[section 7.3]{hl}). 
 
\begin{proposition}{\label{tech}} 
Let $(S,\cal A)$ be a Picard general $K3$ surface,  
with $\cal A^2=2(n-1)$, and let $\cal E\rar S$ be a stable vector bundle  
of rank $r\ges 2$, and $c_1(\cal E)=0$.  
Let $\hat C\srel{\jmath}{\hra} S$ be a reduced and irreducible,  
nodal curve having $\d$ double points, with $\hat C\in|d\cal A|$.  
We denote by $C\srel{\nu}{\rar}\hat C$ its normalization, and   
by $C\xrar{u:=\jmath\circ\nu} S$ the composed morphism.  
 
Suppose that one of the conditions below are satisfied:  
\begin{enumerate} 
\item $d=1$ or $r=d=2$, and  
$\d\les\bigl\lfloor\frac{(r-1)(n-1)d^2-1}{r}\bigr\rfloor 
-c_2(\cal E)$, or 
\item $r\ges 2$ and $d\ges 2$, with $(r,d)\neq(2,2)$,  
and $\d\les 2(n-1)d-c_2(\cal E)- 
\bigl\lfloor\frac{r(n-1)+1}{r-1}\bigr\rfloor$.  
\end{enumerate} 
Then $H^0(C,u^*\cal E^\vee)=0$.  
\end{proposition} 
 
\begin{proof} 
Suppose that there is a non-zero section  
$\cal O_C\srel{s}{\rar}u^*\cal E^\vee$. Then there is an effective  
divisor $\Delta$ on $C$, such that $s$ extends to a monomorphism  
of vector bundles $\cal O_C(\Delta)\rar u^*\cal E^\vee$;  
equivalently, we obtain an epimorphism of vector bundles\\[1ex] 
\centerline{  
$u^*\cal E\srel{q}{\lar}Q:=\cal O_C(-\Delta),\;$ with $\deg_CQ\les 0$. 
}\\[1ex]  
The direct image $\nu_*Q\rar\hat C$ is a torsion free sheaf of rank  
one. We denote by $\hat Q:=\Img(\nu_*q)$ appearing in the diagram\quad  
$\xymatrix{ 
\cal E\otimes\cal O_{\hat C}\ar[r]\ar@/_3ex/[rr]_-{\nu_* q} 
& 
\cal E\otimes \nu_*\cal O_C=\nu_*(u^*\cal E)\ar[r] 
& 
\nu_*Q. 
}$\\  
One can prove that $\hat Q\rar\hat C$ is still a torsion free sheaf  
of rank one, and there is $0\les\d'\les\d$ such that  
\begin{eqnarray*} 
\deg_{\hat C}\hat Q=\deg_C Q+\d'; 
\;\;\text{equivalently,}\;\;  
\chi(\hat Q)=\d'+\deg_C Q-(n-1)d^2. 
\end{eqnarray*} 
One obtains a natural epimorphism of sheaves  
$\veps:\cal E\rar \jmath_*\hat Q$ as follows: 
$$ 
\xymatrix@C=3em{ 
\cal E\ar@{->>}[r]\ar@/_3ex/[rr]_-\veps 
& 
\cal E\otimes\cal O_{\hat C}\ar@{->>}[r] 
& 
\jmath_*\hat Q 
} 
$$ 
We denote by $\cal G:=\Ker(\veps)$; it is a locally free sheaf (a vector  
bundle) of rank $r$ over $S$.  
Using \cite[proposition 5.2.2]{hl}, we compute its numerical invariants:  
\begin{eqnarray*} 
c_1(\cal G)=-[\hat C]=-dA,\quad 
\begin{array}[t]{ll} 
c_2(\cal G)&=c_2(\cal E)+(n-1)d^2+\chi(\hat Q)\\  
&=c_2(\cal E)+\d'+\deg_CQ\les c_2(\cal E)+\d, 
\end{array} 
\\  
\begin{array}[t]{rl} 
\Delta(\cal G):= 
2rc_2(\cal G)-(r-1)c_1^2(\cal G) 
&= 
2r\cdot\left[ 
c_2(\cal G)-\frac{r-1}{r}\cdot(n-1)d^2 
\right] 
\\[1ex]  
&\les  
2r\cdot\left[ 
c_2(\cal E)+\d-\frac{r-1}{r}\cdot(n-1)d^2 
\right]. 
\end{array} 
\end{eqnarray*}  
The hypothesis implies that $\Delta(\cal G)<0$. Therefore  
\cite[theorem 7.3.4]{hl} implies the existence of a subsheaf  
$\cal G'\subset\cal G$ of rank $r'$, with torsion free quotient,  
such that:  
\begin{equation}{\label{xi}} 
\xi_{\cal G',\cal G}:= 
\frac{c_1(\cal G')}{r'}-\frac{c_1(\cal G)}{r}>0, 
\quad\text{and}\quad 
\xi_{\cal G',\cal G}^2\ges \frac{-\Delta(\cal G)}{r^2(r-1)}. 
\end{equation} 
The sheaf $\cal G'$ is also contained in $\cal E$, which is stable  
by hypothesis, hence  $c_1(G')\cdot\cal A<0$.  
Since $S$ is a Picard general $K3$ surface,  
it follows that $c_1(\cal G')=-mA$ with $m\ges 1$.   
 
Further, the inequality $\xi_{\cal G',\cal G}>0$ implies  
$$ 
0<r'd-rm\;\;\Rightarrow\;\; 
1\les (r-1)d-rm\;\;\Rightarrow\;\; 
m\les\frac{(r-1)d-1}{r}. 
$$  
In particular $\frac{r+1}{r-1}\les d$.  
For $d=1$ and for $r=d=2$, this gives already a contradiction, coming  
from the assumption that $u^*\cal E^\vee\rar C$ has non-zero sections.  
 
In higher degrees, we must go further a little bit, and use the  
second inequality in \eqref{xi}:\\[1ex]  
\centerline{$ 
\begin{array}{ll} 
&\disp 
\left(\frac{d}{r}-\frac{1}{r-1}\right)^2\cdot 2(n-1) 
\ges  
\left(\frac{d}{r}-\frac{m}{r'}\right)^2\cdot 2(n-1) 
\ges 
\frac{2(r-1)(n-1)d^2-r(c_2(\cal E)+\d)}{r^2(r-1)} 
\\[2ex] 
\Rightarrow&\disp 
\d\ges2(n-1)d-c_2(\cal E)-\frac{r(n-1)}{r-1}. 
\end{array}  
$}\\[1ex]  
This inequality contradicts again our hypothesis.  
\end{proof} 
 
\begin{remark} 
Let $\cal E\rar S$ be as above, and suppose that $r=2$. Then the  
proof of the proposition shows that actually $u^*\cal E\rar C$ is a  
stable vector bundle, since we have used only that $\deg_CQ\les 0$.  
 
For $r\ges 3$, by applying the result to the exterior powers  
$\overset{\rho}{\bigwedge}\,\cal E$, $\rho=1,\ldots,r-1$  
(which are still stable), it follows that $u^*\cal E^\vee\rar C$  
is a stable vector bundle itself, as soon as the number $\d$ of  
nodes is small enough.  
However, the formula for the upper bound of the number of nodes  
is lengthy, and we did not include it here.  
\end{remark}

The case when $\cal E$ is the tangent bundle $\cal T_S$ of $S$ plays  
a privileged role for proving the rigidity of nodal curves on $K3$  
surfaces. In this case, the previous theorem becomes:  
  
\begin{corollary}{\label{cor:tang}} 
Suppose that $S$ and $u:C\rar S$ are as in proposition \ref{tech}.  
If either  
\begin{enumerate} 
\item $d=1$ and  
$\d\les\d_\mx(n,1):=\bigl\lfloor\frac{n}{2}\bigr\rfloor-25\;$  
{\rm(}hence $\,n\ges50\,${\rm)}, or\smallskip 
  
\item $d=2$ and $\d\les\d_\mx(n,2) := 2n-27\;$  
{\rm(}hence $\,n\ges 14\,${\rm)}, or\smallskip 
  
\item $d\ges 3$ and $\d\les\d_\mx(n,d) := 2(n-1)(d-1)-25\;$  
{\rm(}hence $\,(n-1)(d-1)\ges 13\,${\rm)}, 
\end{enumerate} 
then $u^*\cal T_S\rar C$ is a stable, rank two bundle, and therefore  
$H^0(C,u^*\cal T_S)=0$.  
\end{corollary}

\begin{remark}{\label{rmk:bounds}}  
Notice that for degree one, nodal curves on Picard general $K3$  
surfaces, the upper bound $\d_\mx(n,1)$  
appearing in the proposition above basically equals half of the  
arithmetic genus of the hyperplane section. 
 
In this case $\bigl\lfloor\frac{n}{2}\bigr\rfloor-25$ must be  
positive, and therefore the rigidity result holds for $n\ges 50$.  
This bound is weaker than the (optimal) bound $n\ges 13$ obtained  
by Mukai in \cite{mu}. 
\end{remark} 
 
The upper bounds on the number of nodes obtained in corollary  
\ref{cor:tang} are unlikely to be optimal.  
We are unable to address the following:  
 
\begin{question}{\label{issue:h0}} 
Suppose that $S$, and $u:C\rar S$ are as above, and that the genus of  
$C$ is at least two. Is it true that $u^*\cal T_S\rar C$ has no section? 
\end{question} 
 
A positive answer would allow to extend the rigidity results obtained  
in section \ref{sect:rigid}.


\section{The rigidity result}{\label{sect:rigid}} 
 
In this section we are going to give a (partial) positive answer  
to the question \ref{issue1}.  
 
\begin{theorem}{\label{thm:rigid}} 
The morphism $\mu:\crl V_{n,\d}^d\rar\crl M_{g(d)-\d}$ is generically  
finite onto its image for all triples $(d,n,\d)$ satisfying $n\ges 13$  
and $\d\les\d_\mx(n,d)$, with $\d_\mx(n,d)$ as in corollary \ref{cor:tang}. 
\end{theorem} 
 
\begin{proof}  
We must prove that for any smooth, quasi-projective curve $\disc$,  
and for any morphism $\disc\srel{U}{\rar}\crl V_{n,\d}^d$ such that  
$\mu\circ U:\disc\rar\crl M_{g(d)-\d}$ is constant, the morphism  
$U$ is constant itself.  
Such a morphism $U$ is equivalent to the following data:  
 
-- a smooth and irreducible curve $C$ of genus $g:=g(d)-\d$;  
 
-- a smooth family $({\euf S,\cal A})\overset{\pi}{\rightarrow}\disc$  
of Picard general $K3$ surfaces;  
 
-- a family of morphisms over the $1$-dimensional base $\disc$ 
\begin{equation}{\label{eqn:CX}} 
\xymatrix@R=1.5em{ 
\mathcal C:=\disc\times C 
\ar[rr]^-{U={(u_t)}_{t\in\disc}}\ar[rd]_-{\pr_\disc}&& 
\euf S={(S_t)}_{t\in\disc}\ar[ld]^-\pi 
\\ 
&\disc& 
} 
\end{equation} 
such that $\hat{C}_{t}:=u_{t}(C)\hookrightarrow S_{t}$ are nodal curves,  
with $\d$ ordinary double points, and $\hat C_t\in|d\cal A_t|$ for all  
$t\in\disc$.  
We must prove that, up to isomorphism, $S_t$ and $u_t$ are independent  
of $t\in\disc$.\smallskip   
 
\nit\unbar{Step 1}\quad First of all, note that we may assume that  
$\cal T_\disc\rar\disc$ is trivializable. Otherwise we cover $\disc$  
with trivializable open subsets. We denote by $\del/\del t$ a trivializing   
section of $\cal T_\disc$.  
 
The differentials of the various morphisms in \eqref{eqn:CX} fit into  
the diagram:  
\begin{equation}{\label{exact-sqn}} 
\xymatrix@R=2em{ 
0\ar[r]& 
\pr_C^*\cal T_C\ar[r]\ar[d]^-{U_*}& 
\cal T_{\disc\times C}=\pr_C^*\cal T_C\oplus\pr_\disc^*\cal T_\disc 
\ar[r]\ar[d]^-{U_*}& 
\pr_\disc^*\cal T_\disc\cong 
\cal O_{\disc\times C}\ar[r]\ar@{=}[d]\ar@{-->}[dl]|-{\;\;s\;\;}&0 
\\  
0\ar[r]& 
U^*\cal T_{\euf S/\disc}\ar[r]& 
U^*\cal T_{\euf S}\ar[r]^{\phantom{xxxxx}\pi_*}& 
\cal O_{\disc\times C}\ar[r]&0 
} 
\end{equation} 
We observe that $s:=U_*\left(\del/\del t\right)$  
is a section of $U^*\cal T_{\euf S}\rar\disc\times C$;  
let $s_t:=s|_{\{t\}\times C}\in H^0(C,u_t^*\cal T_{\euf S})$.  
 
The diagram \eqref{eqn:CX} commutes, and therefore the tangential map  
$\pi_{*}:T\mathcal{S}|_{S_{t}}\rightarrow T_{t}\disc$ sends $s$ into  
the trivializing section $\del/\del t\in H^0(\disc,\cal T_{\disc})$.  
It follows that the second row in \eqref{exact-sqn} is split, that is  
$$ 
U^*\cal T_{\euf S}\cong  
\cal O_{\disc\times C}\oplus U^*\cal T_{\euf S/\disc}. 
$$  
With respect to this splitting $s=(s_0,\bar{s})$, where  
$\bar s\in H^{0}(\disc\times C,U^{*}\cal T_{\euf S/\disc})$.  
By hypothesis $\hat C_t\hra S_t$ are nodal curves for all $t\in\disc$.  
Therefore corollary \ref{cor:tang} implies that  
$u_t^{*}\cal T_{S_t}\rightarrow C$ has no non-trivial sections,  
that is $\bar s|_{\{t\}\times C}=0$ for all $t$.  
We deduce that $\bar s=0$,  or intrinsically  
\begin{equation}{\label{udt}} 
U_*\left(\del/\del t\right)=(s_0,0).  
\end{equation} 
 
\nit\unbar{Step 2}\quad We interpret the result in locally on $\euf S$:  
there are local coordinates $(t,z,w)$ on $\euf S$ such that the morphism  
$U$ is given by\\[1ex]  
\centerline{$ 
U(t,x)=\bigl(t,\,\underbrace{z(t,x),w(t,x)}_{=\,u_t(x)}\,\bigr), 
\;\forall\, t\in\disc,\, x\in C. 
$}\\   
The equality \eqref{udt} becomes:  
$$ 
\rd z_{(t,x)}(\del/\del t)=0  
\quad\text{and}\quad 
\rd w_{(t,x)}(\del/\del t)=0, 
\quad\forall\,(t,x)\in\disc\times C.  
$$ 
It follows that $z(t,x)=z(x)$ and $w(t,x)=w(x)$, meaning that the `vertical'  
component $u_t(x)$ of $U$ is independent of the parameter $t$.  
 
Consider an arbitrary $t_0\in\disc$.  
Suppose that $\hat{x}\in\hat{C}_{t_0}$ is a double point, and let 
$x_{1},x_{2}\in C_{t_0}$ be the corresponding pair of points identified  
by the normalization map $C_{t_0}\srel{u_0}{\rar}\hat C_{t_0}$.  
Then for all $t\in\disc$ holds 
$$ 
u_{t}(x_{1})=\bigl(z(t,x_1),w(t,x_2)\bigr) 
=\bigl(z(t_0,x_1),w(t_0,x_2)\bigr) 
=u_{0}(x_{1})=u_{0}(x_{2})=\ldots=u_{t}(x_{2}),  
$$ 
that is the morphism $u_{t}$ will identify the same pairs of points  
of $C$. Since $t$ was arbitrary, we conclude that the curves  
$\hat{C}_{t}:=u_{t}(C)\hookrightarrow S_{t}$  
are all isomorphic to $\hat{C}:=\hat C_{t_0}$.\smallskip  
 
\nit\unbar{Step 3}\quad The previous step reduces the initial problem  
to the study of deformations of pairs $(S,\hat{C})$, consisting of a  
$K3$ surface $S$, and a nodal curve $\hat{C}\hra S$  
(that is we forget about the normalization $\nu:C\rar\hat C$).  
More precisely, we must prove that for any commutative diagram 
\begin{equation}{\label{eqn:hat-CX}} 
\xymatrix@C=1.5em@R=1.5em{ 
\disc\times\hat C\ar[rr]^-{\jmath\;={(\jmath_t)}_t}\ar[rd]&& 
\euf S={(S_t)}_t\ar[ld]^-\pi 
\\ 
&\disc& 
} 
\quad 
\begin{array}[t]{ll} 
\text{such that }&\hat C\cong\jmath_t(\hat C)\hra S_t, 
\;\forall\,t\in\disc 
\\  
&\text{are nodal curves,} 
\end{array} 
\end{equation} 
the family $(\euf{S},\disc\times\hat C,\jmath)$ is trivial.  
 
The deformations of the pair $(S,\hat C)$ are controlled by the  
locally free sheaf $T_{S}\langle\hat{C}\rangle$, defined similarly  
as in \eqref{deform} (see \cite[section 2]{fkps}). It fits into  
the exact sequence  
$$ 
0\longrightarrow\cal T_{S}(-\hat{C}) 
\longrightarrow  
\cal T_{S}\langle\hat{C}\rangle 
\longrightarrow  
\cal T_{\hat{C}}:=\nu_*\cal T_C\longrightarrow0. 
$$ 
The deformation ${\jmath}$ appearing in \eqref{eqn:hat-CX} keeps the nodal  
curve $\hat{C}$ fixed, as an abstract curve. Therefore the infinitesimal deformation  
induced by ${\jmath}$ corresponds to an element\\[1ex]  
\centerline{$ 
\hat e\in\Ker\! 
\left(H^{1}(S,\cal T_{S}\langle\hat{C}\rangle) 
\rightarrow  
H^{1}(\hat{C},\cal T_{\hat{C}})\right) 
= 
\Img\! 
\left(H^{1}(S,\cal T_{S}(-\hat{C})) 
\rightarrow  
H^{1}(S,\cal T_{S}\langle\hat{C}\rangle)\right). 
$}\\[1ex] 
According to \cite{mu}, $H^1(S,\cal T_{S}(-\hat{C}))=0$ for a general  
$(S,\cal A)$ with $\cal A^2=2(n-1)\ges 24$. It follows that $\hat e=0$,  
which means that the deformation of the pair $(S,\hat C)$ is trivial.  
\end{proof} 
 
\begin{remark} 
The first step of the previous proof can be interpreted and proved  
at the level of the Zariski tangent space of $\crl V_{n,\d}^d$.  
Let $e\in H^1(\tilde S,(\si^*\cal T_S)\langle C\rangle)$ be  
the element corresponding to the deformation \eqref{eqn:CX}.  
By diagram chasing in  \eqref{deform} at the level of  
the long exact sequences in cohomology, we obtain the following  
(self-explanatory) diagram:  
\begin{equation*} 
\xymatrix@C=0em@R=2em{ 
\exists!& 
f\ar@{|->}[d]\ar@{=}[rrr] 
& 
&& 
f\ar@{|->}[d]& 
\in H^1\bigl(\tilde S,(\si^*\cal T_S)(-C)\bigr) 
\ar@{>->}[d]^-{\text{It is {\it injective} by corollary \ref{cor:tang}.}} 
\\  
&e\ar@{|->}[d]\ar@{:>}@/^2ex/[u]^-{H^0(C,\cal T_C)\,=0} 
\ar@{|->}@/_3ex/@<-1ex>[rrr] 
& 
\in H^1\bigl(\tilde S,(\si^*\cal T_S)\langle C\rangle\bigr) 
&& 
e_0& 
\in H^1(\tilde S,\si^*\cal T_S)\phantom{MM} 
\\  
&0& 
\in H^1(C,\cal T_C)\phantom{MMM}&&& 
} 
\end{equation*} 
More precisely, we have the inclusion\\[.5ex]  
\centerline{$ 
\Ker\bigl(H^1(r)\bigr) 
\srel{\raise.5ex\hbox{\tiny$H^0(C,\cal T_C)=0$}}{=\kern-.5ex=} 
H^1\bigl(\tilde S,(\si^*\cal T_S)(-C)\bigr) 
\srel{\hbox{\tiny$H^0(C,u^*\cal T_S)=0$}}{\subset} 
H^1(\tilde S,\si^*\cal T_S)\cong H^1(S,\cal T_S) 
$}\\[1ex] 
which shows that we can identify the deformation \eqref{eqn:CX} with the  
induced infinitesimal deformation of $S$, corresponding to $e_0=\pi_*(e)$.  
This is the differential theoretic counterpart of the claim that the  
section $s$ appearing in the proof of \ref{thm:rigid} has the form  
$(s_0,0)$. 
 
The third step can be proved by differential methods too. However  
the second step is not proved at the tangential level: it uses effectively  
the fact that we are considering a $1$-dimensional deformation.  
\end{remark} 
 
As a byproduct we obtain: 
 
\begin{corollary} 
Suppose that $S$ is a general $K3$ surface with $\Pic(S)=\mbb Z\cal A$,  
$\cal A^2=2(n-1)\ges 24$. Let $\hat C\hra S$ be a nodal curve of  
degree $d$ with $\d$ nodes, such that $\d\les\d_\mx(n,d)$; we denote by  
$\euf N$ the set of nodes of $\hat C$, and let  
$\cal I_{\euf N}\subset\cal O_S$ be their ideal sheaf. Then holds:  
$$ 
\begin{array}{ll} 
H^1\bigl(S,\Omega^1_S(\hat C)\otimes\cal I_{\euf N}\bigr)=0,\; 
&\text{or equivalently}\\[1.5ex]  
H^0\bigl(S,\Omega^1_S(\hat C)\bigr) 
\lar\underset{\hat x\in\euf N}{\bigoplus}\Omega^1_{S,\hat x}\; 
&\text{is surjective.} 
\end{array} 
$$  
\end{corollary} 
 
\begin{proof} 
Simply consider the long exact sequence in cohomology corresponding to  
the first column in \eqref{deform}, and use the fact that $H^1(r)$ is injective. 
\end{proof}


\section{Applications to the Wahl map}{\label{sect:wahl}} 

The Wahl map for curves has been considered for the first time in 
\cite{wa}. The surjectivity of the Wahl map represents an obstruction 
for embedding a smooth curve into a $K3$ surface. For an overview of 
these results, and for further generalizations, we invite the reader 
to consult \cite{wa2}. Here we recall only those notions which are 
necessary for this article. 

Suppose that $\cal L\rar V$ is a line bundle over some variety $V$.  
The Wahl map is by definition 
\begin{equation}{\label{wahl-L}} 
w_{\cal L}: 
\overset{2}{\hbox{$\bigwedge$}}\, 
H^0(V,\cal L)\rar H^0(V,\O^1_V\otimes\cal L^2), 
\quad 
s\wedge t\lmt s\rd t-t\rd s. 
\end{equation} 
Equivalently, the Wahl map is the restriction homomorphism  
$H^0({\rm res}_\Delta)$ at the level of sections, induced by  
the exact sequence\\[1ex]  
\centerline{$ 
0\rar 
\cal I_{\Delta_V}^2\otimes(\cal L\boxtimes\cal L)\lar 
\cal I_{\Delta_V}\otimes(\cal L\boxtimes\cal L)\srel{{\rm res}_\Delta}{\lar} 
\underbrace{(\cal I_{\Delta_V}/\cal I_{\Delta_V}^2)}_{\cong\;\O^1_V} 
\otimes\cal L^2\rar 0. 
$}\\[.5ex]  
Much attention has been payed to the case when $V=C$ is a smooth projective  
curve, and $\cal L=\cal M=\cal K_C$, where $\cal K_C$ is the canonical line  
bundle of $C$. The importance of the map  
$$ 
w_C: 
\overset{2}{\hbox{$\bigwedge$}}\,H^0(C,\cal K_C)\rar H^0(C,\cal K_C^3) 
$$ 
relies in the fact that it gives an obstruction to realize the curve $C$ as  
a hyperplane section of a $K3$ surface. More precisely: 
 
\begin{theorem}{\label{thm:w+b}}\  
\begin{enumerate} 
\item {\rm (}see \cite{cmh}{\rm )}  
The Wahl map $w_C$ is surjective for a {\em general} curve $C$ of genus  
at least $12$.  
 
\item {\rm (}see \cite{wa,bm}{\rm )}  
Suppose that $C\hra S$ is a smooth hyperplane section of some  
$K3$ surface. Then the Wahl map is {\em not} surjective. 
\end{enumerate} 
\end{theorem} 
 
In other words, a generic {\em smooth} curve $C\in\crl M_g$  
can not be realized as a hyperplane section of any $K3$ surface,  
as soon as $g\ges 12$.  
The surjectivity of the Wahl map is an obstruction for embedding  
a smooth curve into a $K3$ surface.  
 
Remarkably enough, nodal curves escaped to the attention. Recently, in  
\cite[Question 5.6]{fkps}, the authors asked whether there is an analogous  
obstruction for embedding nodal curves. We will apply the estimates  
obtained in section \ref{sct:0} to prove the following result: 
 
\begin{theorem}{\label{thm:wahl}} 
Let $S$ be a Picard general $K3$ surface, and let $\hat C\hra S$ be a  
nodal curve of degree $d$ with $\d$ nodes, with  
$\d\les\min\left\{\d_\mx(n,d),\frac{(n-1)d^2-1}{3}\right\}$.  
Then the Wahl map\\   
\centerline{$ 
w_C:\hbox{$\overset{2}\bigwedge$}\, H^0(C,\cal K_C)\rar H^0(C,\cal K_C^3) 
$}\\[1ex]  
is not surjective. 
{\rm(} See corollary \ref{cor:tang} for the definition of $\d_\mx(n,d)$. {\rm)} 
\end{theorem} 
 
\nit We remark that for $n\ges 146$, the minimum between the two numbers  
above is  
$$ 
\begin{cases} 
\frac{(n-1)d^2-1}{3}&\text{for }d=1,\ldots,4; 
\\  
2(n-1)(d-1)-25&\text{for }d\ges 5. 
\end{cases} 
$$ 
The proof of the theorem is inspired from \cite{bm}, but contains several  
modifications needed to include the double points.  
Let us recall from {\it loc.\,cit.} that the proof of  
\ref{thm:w+b}(ii) is based on the study of the diagram  
\begin{equation}{\label{smooth-fc}} 
\xymatrix@R=2em{ 
\underset{}{\srel{2}{\bigwedge}} 
\,H^0\bigl(S,\cal O_S(C)\bigr) 
\ar[r]^-{w_S}\ar[d]_-\rho 
& 
H^0\bigl(S,\Omega^1_S\otimes\cal O_S(2C)\bigr) 
\ar[d]^-{\rho_1}\ar[r] 
& 
H^0\bigl(C,\Omega^1_S\bigr|_C\otimes\cal K_C^2\bigr) 
\ar[dl]^-b 
\\  
\underset{}{\srel{2}{\bigwedge}} 
\,H^0\bigl(C,\cal K_C\bigr) 
\ar[r]^-{w_C} 
&  
H^0(C,\cal K_C^3).& 
} 
\end{equation} 
The surjectivity of $w_C$ implies the surjectivity of $b$, and one  
proves that this is impossible.  
 
The main difficulty for extending this proof to the nodal case is that  
of finding appropriate substitutes for the cohomology groups appearing  
in \eqref{smooth-fc}.  
Since this task is rather computational, and is based on diagram chasing,  
it has been deferred to appendix \ref{diag-chase}.  
 
Now we introduce some notations which will be used in the proof  
of theorem \ref{thm:wahl}. We denote by $\cal A\rar S$ the ample  
generator of $\Pic(S)$.  
Let $\hat C\hra S$ be a nodal curve of degree $d$ with $\d$ nodes,  
and let $\euf N=\{\hat x_1,\ldots,\hat x_\d\}\subset S$ be its nodes.  
We denote by $C\srel{\nu}{\rar}\hat C$ the normalization, and  
by $x_{1,1},x_{1,2},\ldots,x_{\d,1},x_{\d,2}\in C$ the pre-images  
by $\nu$ of $\hat x_1,\ldots,\hat x_\d$ respectively.  
Let $\tilde S:={\rm Bl}_{\euf N}(S)\srel{\si}{\rar}S$ be the  
blow-up of $S$ at the nodes of $\hat C$, and $E=E_1+\ldots+E_\d$  
be the exceptional divisor in $\tS$. Then the diagram 
$$ 
\xymatrix@C=5em{ 
C\;\ar[d]_-\nu\ar@{^{(}->}[r]^{\tilde u}\ar[dr]^-u& 
\tilde S\ar[d]^-\si 
\\  
\hat C\;\ar@{^{(}->}[r]^{\jmath}& 
S 
} 
$$ 
is commutative, and $\tilde u$ is an embedding. Note that the divisor  
$\Delta:=x_{1,1}+x_{1,2}+\ldots+x_{\d,1}+x_{\d,2}$ equals  
$C\cdot E$, hence $\cal O_C(\Delta)=\cal O_S(E)|_C=\cal K_\tS\bigr|_C$.  
We deduce the existence of the short exact sequence 
\begin{equation}{\label{tang-sqn}} 
0\rar\cal T_C(\Delta) 
\rar 
\Omega^1_\tS\otimes\cal O_C 
\rar 
\cal K_C\rar0, 
\;\text{ and also that }\;\cal K_C\cong\si^*\cal L(-E). 
\end{equation} 
 
\begin{lemma}{\label{non-split}} 
Let $S,C$, and $u:C\rar S$ be as above.  
If $\d\les\min\left\{\d_\mx(n,d),\frac{(n-1)d^2-1}{3}\right\}$,  
then the exact sequence \eqref{tang-sqn} is not split.  
\end{lemma} 
 
\begin{proof} 
Let us assume that $\tilde u^*\cal T_\tS\cong\cal T_C\oplus\cal K_C(-\Delta)$.  
Then it follows from the diagram 
$$ 
\xymatrix@R=1.7em@C=1.5em{ 
0\ar[r]& 
\cal T_C 
\ar[r]\ar@{=}[d]& 
\tilde u^*\cal T_\tS 
\ar[r]\ar@{}[d]|-{\hbox{\Large$\cap$}}& 
\cal K_C(-\Delta) 
\ar[r]\ar@{}[d]|-{\hbox{\Large$\cap$}}& 
0 
\\  
0\ar[r]& 
\cal T_C\ar[r]& 
u^*\cal T_S\ar[r]& 
\cal K_C\ar[r]& 
0 
} 
$$ 
that $H^0\bigl(C,\cal K_C(-\Delta)\bigr)\subset H^0\bigl(S, u^*\cal T_S\bigr)$.  
But the Riemann-Roch formula implies that  
$$ 
h^0\bigl(C,\cal K_C(-\Delta)\bigr)\ges  
\deg_C\cal K_C(-\Delta)-\bigl(g(C)-1\bigr) 
=\bigl(g(C)-1\bigr)-2\d=g(d)-1-3\d\ges 1. 
$$ 
This contradicts corollary \ref{cor:tang}.  
\end{proof}

We denote by $\cal L:=\cal O_S(C)\cong\cal A^d$.  
The Wahl maps of $C$ and $\tS$ fit into the following commutative  
diagram, which is important in the subsequent constructions: 
\begin{equation}{\label{eqn-fc}} 
\xymatrix@R=2.5em{ 
\underset{}{\srel{2}{\bigwedge}} 
\,H^0\bigl(\tS,\si^*\cal L(-E)\bigr) 
\ar[r]^-{w_\tS} 
\ar@{->>}[d]_-\rho^-{\text{\begin{tabular}{l} 
surjective\\[0ex] by lemma \ref{lm:help} 
\end{tabular}}} 
& 
H^0\bigl(\tS,\Omega^1_\tS\otimes\si^*\cal L(-2E)\bigr) 
\ar[d]^-{\rho_1} 
\\  
\underset{}{\srel{2}{\bigwedge}}\,H^0\bigl(C,\cal K_C\bigr) 
\ar[r]^-{w_C} 
&  
H^0(C,\cal K_C^3). 
}\end{equation} 
 
\begin{proof}(of theorem \ref{thm:wahl})  
We assume that there is a curve $\hat C\hra S$ such that the Wahl map  
of its normalization is surjective.  We define\\   
\centerline{$ 
R(C,\Delta):=w_C^{-1}\bigl(\;H^0\bigl(C,\cal K_C^3(-\Delta)\bigr)\;\bigr) 
\subset{\overset{2}{\hbox{$\bigwedge$}}}\;H^0(C,\cal K_C), 
$}\\[1.5ex]  
and denote by  
$w_{C,\Delta}:R(C,\Delta)\rar H^0\bigl(C,\cal K_C^3(-\Delta)\bigr)$  
the restriction of the Wahl map to it.  
Since $w_C$ is surjective, $w_{C,\Delta}$ is surjective too.  
Furthermore, we define\\[.5ex]  
\centerline{$ 
R:=\rho^{-1}\bigl(R(C,\Delta)\bigr) 
\subset\;\srel{2}{\hbox{$\bigwedge$}}H^0\bigl(\tS,\si^*\cal L(-E)\bigr) 
$}\\[1ex] 
In the appendix we will construct the cube \eqref{cube}.  
Its rear face gives us the diagram  
\begin{equation}{\label{nodal-fc}} 
\xymatrix{ 
R\ar[r] 
\ar@{->>}[d]_{\rho_\Delta}& 
H^1\bigl(\tS,\Omega^1_\tS\otimes\si^*\cal L^2(-3E)\bigr) 
\ar[r]\ar[d]& 
H^0\bigl(C,\Omega^1_\tS\bigr|_C\otimes\cal K_C^2(-\Delta)\bigr) 
\ar@{->>}[dl]^-b 
\\  
R(C,\Delta)\ar@{->>}[r]^-{w_{C,\Delta}}& 
H^0\bigl(C,\cal K_C(-\Delta)\bigr) 
&} 
\end{equation} 
which will be the substitute for the diagram \eqref{smooth-fc}  
in the case of nodal curves.  
 
Indeed, the surjectivity of $\rho_{\Delta}$, and of $w_{C,\Delta}$  
implies the surjectivity of the homomorphism $b$. But $b$ is the  
restriction homomorphism at the level of sections in the exact sequence  
\begin{equation}{\label{sqn:b}} 
0\rar\cal K_C\rar\Omega^1_\tS\bigr|_C\otimes\cal K_C^2(-\Delta) 
\rar\cal K_C^3(-\Delta)\rar 0, 
\end{equation} 
obtained by tensoring \eqref{tang-sqn} with $\cal K_C^2(-\Delta)$.  
Therefore the boundary map\\[1ex]  
\centerline{$ 
\partial:H^0\bigl(C,\cal K_C^3\bigr)\rar  
H^1\bigl(C,\cal K_C\bigr) 
$}\\[1ex] 
vanishes. By applying \cite[Lemme 1]{bm}, we deduce that the sequence  
\eqref{sqn:b} is split, hence \eqref{tang-sqn} is split too.  
This contradicts the lemma \ref{non-split}. 
\end{proof}


\appendix

\section{Diagram chasing}{\label{diag-chase}} 
 
In this section we continue to use the notations introduced in  
section \ref{sect:wahl}.  
 
\begin{lemma}{\label{lm:help}} 
The restriction homomorphisms\\[1ex]  
\centerline{$ 
H^0\bigl(\tS,\si^*L(-E)\bigr)\rar H^0\bigl(C,\cal K_C\bigr) 
\;\text{ and }\;  
H^0\bigl(\tS,\si^*L(-2E)\bigr)\rar H^0\bigl(C,\cal K_C(-\Delta)\bigr) 
$}\\[1ex]  
are both surjective.  
\end{lemma} 
 
\begin{proof} 
We have the exact sequence over $\tS$:\quad  
$0\rar\si^*\cal L^{-1}(2E)\rar\cal O_\tS\rar\cal O_C\rar 0$.  
 
\nit(i) The first statement is obtained by tensoring it by $\si^*\cal L(-E)$,  
and using that\\[1ex]  
\centerline{$ 
H^1\bigl(\tS,\cal O_\tS(E)\bigr)= H^1\bigl(\tS,\cal K_\tS\bigr)=0. 
$}\\[1ex]  
\nit(ii) The second statement is obtained by tensoring the exact sequence  
by $\si^*\cal L(-2E)$, and using that $H^1\bigl(\tS,\cal O_\tS\bigr)=0.$ 
\end{proof} 
 
There is a natural restriction homomorphism  
$\Omega^1_{\tS}\srel{\res_C}{\lar}\cal K_C$ which is surjective, and  
its kernel $\cal F:=\Ker(\res_C)$ is a locally free sheaf (a vector bundle)  
over $\tS$ of rank two. The following commutative diagram is essential  
for the proof of theorem \ref{thm:wahl}:  
\begin{equation}{\label{cde}} 
\xymatrix@R=1.7em{ 
&0\ar[d]&0\ar[d]&0\ar[d]& 
\\  
0\ar[r]& 
\cal F\otimes\si^*\cal L^2(-3E)\ar[r]\ar[d]& 
\Omega^1_\tS\otimes\si^*\cal L^2(-3E)\ar[r]^-{\rho_{1,\Delta}}\ar[d]& 
\cal K_C^3(-\Delta)\ar[r]\ar[d]& 
0 
\\ 
0\ar[r]& 
\cal F\otimes\si^*\cal L^2(-2E)\ar[r]\ar[d]& 
\Omega^1_\tS\otimes\si^*\cal L^2(-2E)\ar[r]^-{\rho_1}\ar[d]& 
\cal K_C^3\ar[r]\ar[d]& 
0 
\\ 
0\ar[r]& 
\cal F\otimes\cal O_E\ar[r]\ar[d]& 
\Omega^1_E\otimes\cal O_E\ar[r]^-{\res_C}\ar[d]& 
\underset{j=1}{\overset{\d}{\bigoplus}}\; 
\cal K_{C,x_{j,1}}\oplus \cal K_{C,x_{j,1}}\ar[r]\ar[d]& 
0 
\\  
&0&0&0& 
} 
\end{equation} 
Actually the whole proof is based on the careful analysis of this diagram.  
 
Every vector bundle on the projective line splits into the direct sum of 
line bundles. Hence the restriction of $\Omega^1_\tS$ to each component 
$E_j$, $j=1,\ldots,\d$, of the exceptional divisor $E$ is the direct sum 
of line bundles. In fact   
$$ 
\Omega^1_\tS\otimes\cal O_{E_j} 
=\Omega^1_{E_j}\oplus\cal N_{E_j|\tS}^\vee 
=\cal O_{E_j}(-2)\oplus\cal O_{E_j}(1). 
$$ 
Therefore   
$ 
\Omega^1_\tS\otimes\cal O_E=\Omega^1_E\oplus\cal N_{E|\tS}^\vee 
=\cal O_E(-2)\oplus\cal O_E(1)$, where we use the shorthand notation  
$$ 
\cal O_{E}(-2):= 
\overset{\d}{\underset{j=1}{\hbox{$\bigoplus$}}}\;\cal O_{E_j}(-2), 
\quad\text{and}\quad 
\cal O_{E}(1):=\overset{\d}{\underset{j=1}{\hbox{$\bigoplus$}}}\;
\cal O_{E_j}(1). 
$$  
Since the homomorphism ${\rm res}_C$ is the restriction of $1$-forms on $\tS$  
to $1$-forms on $C$, we deduce from the last line in \eqref{cde} that  
\begin{equation}{\label{fe}} 
\cal F\otimes\cal O_E=\cal O_E(-2)\oplus\cal O_E(-1). 
\end{equation} 
 
\begin{lemma}{\label{lm-help1}} 
{\rm (i)} $H^0(\cal F\otimes\cal O_E)=0\;$  and   
$\;H^1(\cal F\otimes\cal O_E)=H^1(\cal O_E(-2)) 
\cong\underbrace{\mbb C\oplus\ldots\oplus\mbb C}_{\d\;\text{ times}}$; 
 
\nit{\rm (ii)}  
$H^0\bigl(\tS,\cal F\otimes\si^*\cal L^2(-3E)\bigr) 
\srel{\cong}{\lar}  
H^0\bigl(\tS,\cal F\otimes\si^*\cal L^2(-2E)\bigr)$  
is an isomorphism; 
 
\nit{\rm (iii)} $H^1\bigl(\tS,\cal F\otimes\si^*\cal L^2(-3E)\bigr) 
\lar H^1\bigl(\tS,\cal F\otimes\si^*\cal L^2(-2E)\bigr)$  
is injective.  
\end{lemma} 
 
\begin{proof} 
(i) It follows from \eqref{fe}.  
 
\nit(ii) and (iii) Consider the long exact sequence in cohomology  
corresponding to the first column in \eqref{cde}.  
The claims follows from (i) above.  
\end{proof} 
 
\nit{\bf Standing hypothesis.}\quad 
From now on we will assume that the nodal curve $\hat C\hra S$ has  
the property that the Wahl map of its normalization\\[1ex]  
\centerline{$ 
w_C:\overset{2}\bigwedge\,H^0(C,\cal K_C)\rar H^0(C,\cal K_C^3) 
$}\\[1ex] 
is surjective. 
 
\begin{lemma}{\label{help-2}} 
\nit{\rm (i)} The homomorphisms  
$$ 
\begin{array}{cc} 
H^1\bigl(\tS,\cal F\otimes\si^*\cal L^2(-2E)\bigr) 
\rar H^1\bigl(\tS,\Omega^1_\tS\otimes\si^*\cal L^2(-2E)\bigr), 
&\;\text{ and} 
\\[1ex] 
H^1\bigl(\tS,\cal F\otimes\si^*\cal L^2(-3E)\bigr) 
\rar H^1\bigl(\tS,\Omega^1_\tS\otimes\si^*\cal L^2(-3E)\bigr) 
\end{array} 
$$ 
are injective. 
 
\nit{\rm (ii)} The restriction homomorphisms  
$$ 
\begin{array}{lc} 
\kern2ex 
\rho_1:H^0\bigl(\tS,\Omega^1_\tS\otimes\si^*\cal L^2(-2E)\bigr) 
\rar H^0\bigl(C,\cal K_C^3\bigr), 
&\;\text{and} 
\\[1.5ex]  
\rho_{1,\Delta}:H^0\bigl(\tS,\Omega^1_\tS\otimes\si^*\cal L^2(-3E)\bigr) 
\rar H^0\bigl(C,\cal K_C^3(-\Delta)\bigr) 
\end{array} 
$$ 
are surjective. 
 
\nit{\rm (iii)}  
$\begin{array}[t]{l} 
H^0\bigl(\tS,\Omega^1_\tS\otimes\si^*\cal L^2(-3E)\bigr) 
\cong\rho_{1,\Delta}^{-1}\bigl(\;H^0\bigl(C,\cal K_C^3(-\Delta)\bigr)\;\bigr) 
\\[1ex]  
=\bigl\{ 
s\in H^0\bigl(\tS,\Omega^1_\tS\otimes\si^*\cal L^2(-2E)\bigr) 
\mid \rho_1(s)\in H^0\bigl(C,\cal K_C^3(-\Delta)\bigr) 
\bigr\}. 
\end{array}$ 
\end{lemma} 
 
\begin{proof} 
(i) In the commutative diagram \eqref{eqn-fc} the homomorphisms $w_C$  
and $\rho$ are surjective, hence $\rho_1$ is also surjective. We deduce the  
injectivity of the first homomorphism from the second line in \eqref{cde}.  
On the other hand, it follows from \eqref{cde} that we have the  
commutative square  
$$ 
\xymatrix@R=2em@C=7em{ 
H^1\bigl(\tS,\cal F\otimes\si^*\cal L^2(-3E)\bigr) 
\ar[r]\ar[d]^-{\hbox{\scriptsize$\begin{array}{l} 
\rm injective\\ \rm by\;\ref{lm-help1}\rm (ii)\end{array}$}}& 
H^1\bigl(\tS,\Omega^1_\tS\otimes\si^*\cal L^2(-3E)\bigr) 
\ar[d] 
\\  
H^1\bigl(\tS,\cal F\otimes\si^*\cal L^2(-2E)\bigr) 
\ar[r]^-{\rm injective}_-{{\rm as\;} \rho_1{\;\rm surjective}}& 
H^1\bigl(\tS,\Omega^1_\tS\otimes\si^*\cal L^2(-2E)\bigr). 
} 
$$ 
It follows that the upper homomorphism is injective too, as claimed.  
 
\nit{\rm(ii)} The surjectivity of $\rho_1$ has been proved already.  
For the second one, consider the long exact sequence in cohomology  
corresponding to the first line in \eqref{cde}, and use (i) above.  
 
\nit{\rm(iii)} The first two rows of \eqref{cde}, together with (ii)  
above imply that we have the commutative diagram 
$$ 
\xymatrix@R=1.7em@C=1.5em{ 
0\ar[r]& 
H^0\bigl(\tS,\cal F\otimes\si^*\cal L^2(-3E)\bigr) 
\ar[r]\ar[d]^-{\cong}& 
H^0\bigl(\tS,\Omega^1_\tS\otimes\si^*\cal L^2(-3E)\bigr) 
\ar[r]\ar@{}[d]|-{\hbox{\Large$\cap$}}& 
H^0\bigl(C,\cal K_C^3(-\Delta)\bigr) 
\ar[r]\ar@{}[d]|-{\hbox{\Large$\cap$}}& 
0 
\\  
0\ar[r]& 
H^0\bigl(\tS,\cal F\otimes\si^*\cal L^2(-2E)\bigr)\ar[r]& 
H^0\bigl(\tS,\Omega^1_\tS\otimes\si^*\cal L^2(-2E)\bigr)\ar[r]& 
H^0\bigl(C,\cal K_C^3\bigr)\ar[r]& 
0 
} 
$$ 
The claim is a consequence of the fact that the first vertical  
arrow is an isomorphism. 
\end{proof} 
 
We define  
$R(C,\Delta):=w_C^{-1}\bigl(\;H^0\bigl(C,\cal K_C^3(-\Delta)\bigr)\;\bigr) 
\subset{\overset{2}\bigwedge}\;H^0(C,\cal K_C)$, and denote by $w_{C,\Delta}$  
the restriction of the Wahl map to it.  
Then $w_{C,\Delta}:R(C,\Delta)\rar H^0\bigl(C,\cal K_C^3(-\Delta)\bigr)$ is  
surjective because $w_C$ is surjective.  
 
The cohomology groups introduced so far fit into the following  
commutative cube: 
\begin{equation}{\label{cube}} 
\xymatrix@1@-1.7em@R=2.5em{ 
R:=\rho^{-1}\bigl( R(C,\Delta) \bigr)  
\ar[rr]^<(.4){w_{\tS,E}} 
\ar@{->>}[dd]_<(.7){\rho_\Delta}^<(.7){\hbox{\tiny 
\begin{tabular}{l} 
surjective\\ since $\rho$ is so. 
\end{tabular}}} 
\ar[dr]_-\subset 
& & H^0\bigl(\tS,\Omega^1_\tS\otimes\si^*\cal L(-3E)\bigr) 
\ar[dr]^-\subset 
\ar@{->>}'[d][dd]^-{\hbox{\tiny\begin{tabular}{l} 
surjective\\ by lemma \ref{help-2} 
\end{tabular}}}_-{\rho_{1,\Delta}} 
\\  
&\underset{}{\overset{2}\bigwedge}\;H^0\bigl(\tS,\si^*\cal L(-E)\bigr) 
\ar[rr]^<(.2){w_\tS}\ar@{->>}[dd]_<(.15)\rho 
&&H^0\bigl(\tS,\Omega^1_\tS\otimes\si^*\cal L(-2E)\bigr) 
\ar@{->>}[dd]^-{\rho_1} 
\\ 
 R(C,\Delta) 
\ar[dr]_-\subset\ar@{->>}[rr]|!{[ur];[dr]}{\phantom{xxx}}^<(.6){w_{C,\Delta}} 
&&H^0\bigl(C,\cal K_C^3(-\Delta)\bigr) 
\ar[rd]^-{\subset} 
\\ 
&\underset{}{\overset{2}\bigwedge}\;H^0(C,\cal K_C) 
\ar@{->>}[rr]^<(.25){w_C} 
& & H^0(C,\cal K_C^3) 
} 
\end{equation} 
The $'\subset\,'$ signs on various arrows denote inclusions.

\end{document}